\documentclass[12pt]{amsart}
\usepackage{amsfonts}
\usepackage{amsmath}
\usepackage{amssymb}
\usepackage{graphicx}
\usepackage{amscd}
\usepackage{xcolor}
\usepackage{hyperref}

\newtheorem{theorem}{Theorem}
\newtheorem{acknowledgement}[theorem]{Acknowledgement}

\newtheorem{lemma}[theorem]{Lemma}

\newtheorem*{remark}{Remark}

\begin{document}






\begin{abstract}
We obtain pointwise upper bounds on the derivatives of the heat kernel on
Damek-Ricci spaces and we study the $L^{p}$-boundedness of Littlewood-Paley-Stein
operators.
\end{abstract}

\title[Estimates of the derivatives of the heat kernel ]{Littlewood-Paley-Stein operators on Damek-Ricci spaces}
\author{A. Fotiadis and  E. Papageorgiou }
\curraddr{Department of Mathematics, Aristotle University of Thessaloniki,
Thessaloniki 54.124, Greece }
\maketitle





\section{Introduction and statement of the results}

\label{sectionintroduction}

In this article we prove estimates of the derivatives of the heat kernel on a
Damek-Ricci space. Applying these
estimates we prove the $L^{p}$-boundedness of Littlewood-Paley-Stein
operators.

For the statement of our results we need to introduce some notation. A Damek-Ricci space $S$
is a solvable extension $S= A \ltimes N$ of a Heisenberg type group $N$, equipped with an
invariant Riemannian structure. They are named after E. Damek and F. Ricci, who noted that they are harmonic spaces, \cite{DAMRIC}. As Riemannian manifolds, these solvable Lie groups include all symmetric spaces
of the noncompact type and rank one. Most of them are nonsymmetric harmonic manifolds, and provide
counterexamples to the Lichnerowicz conjecture.

Let $n$ be the dimension of $S$ and let $-\Delta_{S}$ denote the associated Laplace-Beltrami operator. The $L^2$ spectrum of the Laplace-Beltrami operator is the half line
$[Q^2/4, +\infty)$, where $Q$ is the homogeneous dimension
of $N$ (see Section \ref{Preliminaries} for more details). Then, the heat kernel
$h_t$ of $S$ is the fundamental solution of the heat equation $\partial_t h_t=-\Delta_S h_t.$ It is a radial convolution kernel on $S$, i.e. $h_t(x,y)=h_t(d(x,y)).$ By realizing the group $S$ as the unit ball in the corresponding Lie algebra $\mathfrak{s}$, the heat kernel can be expressed as a radial function $h_t(r)$, where $r$ is the geodesic distance to the origin in the ball model.

One of the main objectives in the present article is to prove the following
result.

\begin{theorem}
\label{Main result} If $S$ is a Damek-Ricci space, then for
all $\epsilon>0$ and $i\in \mathbb{N}$ there is a constant $c>0$ such that
\begin{equation}  \label{estimates main result}
\left\vert \frac{\partial ^{i}h_{t}}{\partial t^{i}}(r)\right\vert \leq
c\;t^{-\frac{n}{2}-i}e^{-(1-\epsilon )\left(\frac{Q^2}{4} t+\frac{Q}{2} r +\frac{r^{2}}{4t}\right)},
\end{equation}
$\text{ for all } t>0 \text{ and }r\geq 0.$
\end{theorem}

For the proof we use an inductive argument and we obtain a decreasing sequence of upper bounds for the derivatives of the heat kernel.

Denote by $H_{t}=e^{-\Delta_S t}$ the heat semigroup on $S$. We denote with $\mathbb{N}$ the set of natural numbers, including $0$. Fix $i\in \mathbb{%
N}$. Then, for all $\sigma \geq 0$ we consider as in \cite{AN1,LIS}, the $%
\sigma $-maximal operator
\begin{equation} \label{sigmamax}
{H}_{\sigma, max }(f)=\sup_{t>0}e^{\sigma t}t^{i} \left| \frac{\partial ^{i}}{%
 \partial t^{i}}H_{t}f \right| ,
\end{equation}%
and the Littlewood-Paley-Stein operator
\begin{equation}\label{Hsigma}
{H}_{\sigma }(f)(x)=\left( \int_{0}^{\infty }e^{2\sigma t}\left( t^{i} \left\vert
\frac{\partial ^{i}}{\partial t^{i}}H_{t}f(x)\right\vert ^{2}+\|\nabla
_{x}H_{t}f(x)\|
^{2}\right) \frac{{dt}}{t}\right) ^{1/2}.
\end{equation}%

Next, we apply Theorem \ref{Main result} in order to prove the following result.

\begin{theorem}
\label{maximal} Suppose that $S$ is a Damek-Ricci space. Then, the
operators ${H}_{\sigma, max}$ and ${H}_{\sigma}$ are bounded on $L^{p} (S),\;
p\in (1,\infty),$ provided that
\begin{equation}\label{sigmacondition}
\sigma < Q^2/ pp'.
\end{equation}
\end{theorem}

Note that the aforementioned result coincides with that of \cite[pp.276, 279]{AN1} when $S$ is a noncompact symmetric space of rank one.

There is a very rich and long literature concerning heat kernel estimates in various geometric contexts. See for example \cite{ALX, ANJ, ANO, GR1}, and the references therein. Davies and Mandouvalos in \cite{DAVMAN} obtained optimal estimates of the
heat kernel in hyperbolic spaces and Anker, Damek and Yacoub in \cite%
{AND},
obtained estimates of the heat kernel in the case of Damek-Ricci spaces. The time derivative of the heat kernel has been estimated in \cite{MANTSE} for hyperbolic
spaces, in \cite{LIS} for some manifolds of exponential volume growth, and in \cite{GR2} in a more general setting.

The Littlewood-Paley-Stein operator was first introduced and studied by Lohou%
\'{e} \cite{LO}, in the case of Riemannian manifolds of non-positive
curvature. In a variety of geometric settings it has been proved that ${H}_{\sigma, max}$ and ${H}_{\sigma}$ are bounded on $L^{p}$, $p\in (1,\infty),$
under some conditions on $\sigma$ (see for example \cite{LO} and \cite{AN1}).

Let us now outline the organization of the paper. In Section \ref%
{Preliminaries} we recall some preliminaries about Damek-Ricci spaces and
the heat kernel. In Section \ref{Proof of
the main result} we prove Theorem \ref{Main result} and in Section \ref{Functions
of the Laplacian} we prove Theorem \ref{maximal}.

Throughout this article the different constants will always be denoted by
the same letter $c$. When their dependence or independence is significant,
it will be clearly stated.

\section{Preliminaries}

\label{Preliminaries}

\subsection{Damek-Ricci spaces}
We shall recall
some basic facts on Damek-Ricci spaces. For details,
see \cite{DAMEK, DAMRIC}.
A Damek-Ricci space
$S$ is a solvable Lie group, equipped with a left-invariant metric. More precisely, $S= A \ltimes N$ is a semi-direct product of $ A \simeq \mathbb{R}$ with a Heisenberg type Lie group $N$.

Let us recall the structure of an $H$-type Lie group. Let $\mathfrak{n}$ be a two-step nilpotent Lie algebra, equipped with an inner product $  \langle \;, \; \rangle$. Let us denote by $\mathfrak{z}$  the center of $\mathfrak{n}$ and by $\mathfrak{v}$ the orthogonal complement of $\mathfrak{z}$ in $\mathfrak{v}$ (so that $[\mathfrak{v},\mathfrak{v}] \subset \mathfrak{z}$), with $k= \text{dim} \mathfrak{z}$ and $m=\text{dim}\mathfrak{v}$.
\\
Let $J_{Z}: \mathfrak{v} \rightarrow \mathfrak{v} $ be the linear map defined by
\[\
\langle J_{Z}X, Y \rangle = \langle Z, [X,Y] \rangle ,
\]
for $X,Y \in \mathfrak{v}, Z \in \mathfrak{z}$.

Then $\mathfrak{v}$ is of \textit{Heisenberg type} if the following condition is satisfied:
\[\
J_Z^2= -|Z|^2I \quad \text{ for every } z \in \mathfrak{z}.
\]

The corresponding connected Lie group $N$ is then called of \textit{Heisenberg type}, and we shall identify $N$ with the Lie algebra $\mathfrak{n}$ via the exponential map:
\begin{eqnarray*}
\mathfrak{v} \times \mathfrak{z}  \rightarrow & N \\
(X,Z)  \mapsto & \exp(X+Z).
\end{eqnarray*}

 Thus, multiplication in $N \equiv \mathfrak{n} = \mathfrak{v} \oplus\mathfrak{z} $  will be
\[\
(X,Z)(X',Z')=(X+X',Z+Z'+\frac{1}{2}[X,X']).
\]

Let $\mathfrak{a}$ denote the Lie algebra of $A$ and $H$ a vector in $\mathfrak{a}$, acting on $\mathfrak{v}$ with eigenvalue $1/2$ and $\mathfrak{z}$ with eigenvalue $1$; we
extend the inner product on $\mathfrak{n}$ to the algebra $\mathfrak{s}=\mathfrak{n} \oplus \mathfrak{a}$, by requiring $\mathfrak{n}$ and $\mathfrak{a}$ to be
orthogonal and $H$ to be a unit vector \cite[p.977]{ANPEVA}. The product in
$S= N \ltimes \mathbb{R}^{*}_{+},$ is given by the rule
\[\
(X,Z,a)(X',Z',a')=(X+a^{1/2}X',Z+aZ'+\frac{1}{2}a^{1/2}[X,X'],aa').
\]
 Then, $S$ is solvable connected Lie group, with Lie algebra $\mathfrak{s}=\mathfrak{v} \oplus \mathfrak{z} \oplus \mathbb{R}$ and Lie bracket
 \[\
 [ (X,Z,l), (X',Z',l')]=( \frac{1}{2} l X' -\frac{1}{2} l' X , lZ'-l'Z+[X,X'] ,0).
  \]
Let us endow $S$ with the left invariant Riemannian metric which agrees with the inner product on $\mathfrak{s}$ at the identity and is induced by
 \[\
 \langle (X,Z,l), (X',Z',l') \rangle= \langle X, X' \rangle + \langle Z, Z' \rangle +ll'
 \]
 on $\mathfrak{s}$. Let $d$ denote the distance induced by this Riemannian
 structure. The Riemannian manifold $(S, d)$ is then called a \textit{Damek-Ricci space}.

 The associated left-invariant Haar measure on $S$ is given by
 \begin{equation} \label{Haar}
 a^{-Q}{dX}{dZ}{\frac{da}{a}},
 \end{equation}
where $Q=\frac{m}{2}+k$ is the \textit{homogeneous dimension} of $N.$

 Similarly to the case of hyperbolic spaces, these general groups $S$ can be realized as the unit ball
 \[\
 B(\mathfrak{s})= \lbrace (X,Z,l) \in \mathfrak{s} :
 |X|^2+|Z|^2+l^2<1  \rbrace
 \]
 in $\mathfrak{s}$, through a \textit{Cayley type transform}:
 \begin{eqnarray*}
  C: \quad S \quad & \rightarrow &\quad B(\mathfrak{s})\\
  \quad x=(X,Z,a) \quad & \mapsto& \quad x'=(X',Z',l'),
 \end{eqnarray*}
 where

$\begin{cases}
	X'= \lbrace (1+a+\frac{1}{4}|X|^2)^2+|Z|^2 \rbrace ^{-1} \; \lbrace (1+a+\frac{1}{4}|X|^2)- J_z \rbrace X, \\
	Z' = \lbrace (1+a+\frac{1}{4}|X|^2)^2+|Z|^2 \rbrace ^{-1} \; 2Z, \\
	
	l' \; = \lbrace (1+a+\frac{1}{4}|X|^2)^2+|Z|^2 \rbrace ^{-1}  \; \lbrace (-1+(a+\frac{1}{4}|X|^2))^2+|Z|^2 \rbrace .
\end{cases} $

\bigskip
In the ball model $B(\mathfrak{s})$, the geodesics passing through the origin are the diameters. The geodesic distance to the origin is given by
\[\
r=d(x',0) =\text{log}\frac{1+\|x'\|}{1-\|x'\|}, \quad \rho= \|x'\|=\text{tanh}\frac{r}{2},
\]
 and the Riemannian volume is
 \begin{eqnarray*}
 dV &= &2^n(1-\rho ^2)^{-Q-1}\rho^{n-1}d\rho d \sigma \\
 &= &2^{m+k}(\text{sinh}\frac{r}{2})^{m+k} (\text{cosh}\frac{r}{2})^{k} d\rho d \sigma,
 \end{eqnarray*}
 where $d \sigma$ denotes the surface measure on the unit sphere $\partial B(\mathfrak{s})$ in $\mathfrak{s}$ and $n=\text{dim}S=m+k+1$.

\label{The heat operator on Damek-Ricci spaces}

\subsection{The heat kernel on Damek-Ricci spaces}\label{The heat kernel}

If $\kappa = \kappa (r) $ is a locally integrable radial function  and $\ast |\kappa |$ denotes the
convolution operator whose kernel is $|\kappa |$, then in \cite[Theorem 3.3]{AND}
  it is proved that
\begin{equation} \label{kunze}
\Vert \ast |\kappa |\Vert _{L^{p}(S)\rightarrow L^{p}(S)}=  \underset{S}{%
\int }{dx}|\kappa (x)|\phi _{i(\frac{1}{p}-\frac{1}{2})Q}(x),
\end{equation}%
where $\phi _{\lambda }$ are the elementary spherical functions, $Q$ is the homogeneous dimension of the Heisenberg type Lie group $N$ and ${dx}$ is the measure in (3). Note that (\ref{kunze}) is a version of the Kunze-Stein phenomenon \cite{HE} which holds true for Damek-Ricci spaces. 

Using polar coordinates on $S$, \cite[p.656]{AND},
\begin{equation}\label{spherical}
\phi _{i(\frac{1}{p}-\frac{1}{2})Q}(r) \asymp
\begin{cases}
e^{-\frac{Q}{p'}r}, \quad \text{if} \quad 1 \leq p < 2, \\
(1+r) e^{-\frac{Q}{2}r}, \quad \text{if} \quad p=2. %
\end{cases}%
\end{equation}

Denote by $h_{t}$ the heat kernel on the Damek-Ricci space $S$. Then, $h_t$ is a radial right-convolution kernel on $S$:
\[\
h_t(x,y)=h_t(d(x,y)).
\]
Then, the following estimate holds:
\begin{equation} \label{anker_est}
h_t(r) \asymp t^{-\frac{3}{2}}(1+r) \left( 1+ \frac{1+r}{t} \right)^{\frac{n-3}{2}} e^{-\frac{Q^2}{4}t-\frac{Q}{2}r-\frac{r^2}{4t}} ,
\end{equation}
for $t>0$ and $r \geq 0$ (see \cite[p.664]{AND} for details).

Consequently, (\ref{anker_est}) implies the upper bound
\begin{equation} \label{heat}
h_t(r)\leq c\;t^{-\frac{n}{2}}(1+t)^{\frac{n-3}{2}}(1+r)^{\frac{n-1}{2}}  e^{-\frac{Q^2}{4}t-\frac{Q}{2}r-\frac{r^2}{4t}}.
\end{equation}

Also, in \cite[p.669]{AND} the authors derive the following gradient estimate:
\begin{equation} 
\|\nabla h_t(r) \| \asymp t^{-\frac{3}{2}}r\left( 1+\frac{1+r}{t} \right)^{\frac{n-1}{2}} e^{-\frac{Q^2}{4}t-\frac{Q}{2}r-\frac{r^2}{4t}},
\end{equation}
which implies
\begin{equation} \label{grad}
\|\nabla h_t(r) \| \leq c\;t^{-\frac{n+2}{2}}(1+t)^{\frac{n-1}{2}}(1+r)^{\frac{n+1}{2}}  e^{-\frac{Q^2}{4}t-\frac{Q}{2}r-\frac{r^2}{4t}}.
\end{equation}

Grigory'an in \cite{GR2} derived Gaussian upper bounds for all time
derivatives of the heat kernel, under some assumptions on the on-diagonal
upper bound for $h_{t}$ on an arbitrary complete non-compact Riemannian
manifold $M$. More precisely, he proves that if there exists an increasing
continuous function $f(t)>0,$ $t>0,$ such that
\begin{equation*}
h_{t}(x,x)\leq \frac{1}{f(t)},\text{ for all }t>0\text{ and }x\in M,
\end{equation*}%
then,
\begin{equation}
\left\vert \frac{\partial ^{i}h_{t}}{\partial t^{i}}\right\vert (x,y)\leq
\frac{1}{\sqrt{f(t)f_{2i}(t)}},\text{ for all }i\in \mathbb{N},\text{ }t>0,%
\text{ }x,y\in M,  \label{Grigory'an}
\end{equation}%
where the sequence of functions $f_{i}=f_{i}(t),$ is defined by
\begin{equation*}
f_{0}(t)=f(t)\text{ and }f_{i}(t)=\int_{0}^{t}f_{i-1}(s)ds,i\geq 1.
\end{equation*}

\section{Proof of Theorem \protect\ref{Main result}}

\label{Proof of the main result}

In this section we give the proof of Theorem \ref{Main result}. More
precisely we shall prove the following estimate: for all $\epsilon >0$ and $%
i\in \mathbb{N},$ there is a $c>0$ such that
\begin{equation}
\left\vert \frac{\partial ^{i}h_{t}}{\partial t^{i}}(r)\right\vert \leq
c\; t^{-\frac{n}{2}-i}e^{-(1-\epsilon )\left(\frac{Q^2}{4} t+\frac{Q}{2} r +\frac{r^{2}}{4t}\right)},
\label{required estimate}
\end{equation}%
for all $r\geq 0$ and all $t>0.$

For the proof of (\ref{required estimate}) we need several lemmata. The
following lemma is technical but important for the proof of Theorem \ref%
{Main result}. It provides a method to obtain estimates for the first derivative of a function, given some upper bounds on the function and its second derivative.

\begin{lemma}
\label{Porper} Let
\begin{equation}
\alpha >\beta,\text{ }D\geq D_{\ast },\text{ }B\geq B_{\ast },\text{ }%
C\geq C_{\ast },  \label{conditions Porper}
\end{equation}%
and assume that for fixed $r\geq 0$, the function $%
f_{r}:(0,+\infty )\rightarrow \mathbb{R}$ satisfies
\begin{equation}
\left\vert f_{r}(t)\right\vert \leq c\;t^{-\alpha }(1+t)^{\beta }(1+r )^{\gamma }e^{-Dt-Br -C{r^{2} }/{(4t)}}  \label{fH1}
\end{equation}%
and
\begin{equation}
\left\vert \frac{d^{2}f_{r}}{dt^{2}}(t)\right\vert \leq c\;t^{-\alpha
-2}(1+t)^{\beta }(1+r )^{\gamma }e^{-D_{\ast
}t-B_{\ast }r -C_{\ast }{r^{2}}/{%
(4t)}}.  \label{fH2}
\end{equation}%
Then, for all $\epsilon \in (0,1)$, there is a constant $c>0$, that does not depend on $r$, $t$, such that for all $%
r\geq 0,$
\begin{equation*}
\begin{split}
\left\vert \frac{df_{r}}{dt}(t)\right\vert \leq & c \; t^{-\alpha -1}(1+t)^{\beta
}(1+r)^{\gamma }\\
&\times e^{-\left( {(D_{\ast }+D)t}/{2}+{(B_{\ast }+B)}r/{2}+{(C_{\ast }+{C\lambda _{\epsilon })}}%
{r^{2}}/{8t}\right) },
\end{split}
\end{equation*}%
where $\lambda _{\epsilon }=\frac{1-\epsilon }{1+\epsilon }.$
\end{lemma}

\begin{proof}
By applying twice the mean value
theorem, one can prove that
\begin{equation}  \label{mvt}
\left|\frac{df_{r}}{dt}(t)\right| \leq \frac{1}{\delta}\left(
\left|f_{r}(t)\right| + \left| f_{r}(t+\delta)\right|\right) + \delta
\sup_{\tau \in (t,t+\delta)}\left|\frac{d^{2}f_{r}}{dt^2}(\tau)\right|, \text{ for all }\delta >0 .
\end{equation}

Note that $t\mapsto t^{-\alpha }(1+t)^{\beta }$ is a decreasing
function of $t$, since $\alpha >\beta ,$ therefore
\begin{equation*}
(t+\delta )^{-\alpha }(1+t+\delta )^{\beta }\leq t^{-\alpha }(1+t)^{\beta }.
\end{equation*}%
Thus (\ref{mvt}) and the estimates (\ref{fH1}) and (\ref{fH2}) imply that
\begin{equation*}
\begin{split}
\left\vert \frac{df_{r}}{dt}(t)\right\vert & \leq c\; \frac{1}{\delta }%
t^{-\alpha }(1+t)^{\beta }(1+r)^{\gamma }e^{-Dt-Br -C{%
r ^{2}}/{4(t+\delta )}}+ \\
& +c \; \delta t^{-\alpha -2}(1+t)^{\beta }(1+r)^{\gamma }e^{ -D_{\ast
}t-B_{\ast }r -C_{\ast }{r^2}/{%
4(t+\delta )} }.
\end{split}%
\end{equation*}
Choose now
\begin{equation*}
\delta= \epsilon t e^{ -{(D-D_{*})t}/{2} - {(B-B_{*})r }/{2} -(C-C_{*}){r^2}/{(8 t)}  },
\end{equation*}
in order to balance the exponential terms. Thus,
\begin{equation}
\begin{split}
\left\vert \frac{df_{r}}{dt}(t)\right\vert & \leq  c\; %
t^{-\alpha -1}(1+t)^{\beta }(1+r)^{\gamma }e^{-(D+D_{\ast })t/{2}-(B+B_{\ast
	})r /{2}}\\
&  \times(  \frac{1}{\epsilon }\; e^{+(C-C_{\ast }) r^{2}/(8t)-Cr^{2}/4(t+\delta )}+ \epsilon \;e^{  -(C-C_{\ast }) r^{2}/{(8t)}-C_{\ast }r^{2}/4(t+\delta )} ).
\end{split}
\label{inequality0}
\end{equation}
From (\ref{conditions Porper}) it follows that $\delta \leq \epsilon t.$
Thus
\begin{equation*}
\frac{1}{2t}-\frac{1}{t+\delta }\leq -\frac{1-\epsilon }{2t(1+\epsilon )}=-%
\frac{\lambda _{\epsilon }}{2t}.
\end{equation*}%
Consequently,
\begin{equation}
\frac{C-C_{\ast }}{2}\frac{r^2}{4t}-C\frac{%
r^{2}}{4(t+\delta )}\leq -\frac{r^{2}}{4t}\frac{C\lambda _{\epsilon }+C_{\ast }}{2},
\label{inequality1}
\end{equation}%
and similarly
\begin{eqnarray}
\frac{C-C_{\ast }}{2}\frac{r^{2}}{4t}+\frac{C_{\ast
}r^{2}}{4(t+\delta )} &\geq &\frac{r^{2}}{4t}\frac{C_{\ast }\lambda _{\epsilon }+C}{2}  \notag \\
&\geq &\frac{r^{2}}{4t}\frac{C\lambda _{\epsilon
}+C_{\ast }}{2}.  \label{inequality2}
\end{eqnarray}%
Thus, from (\ref{inequality0}), (\ref{inequality1}) and (\ref{inequality2})
it follows that
\begin{equation*}
\begin{split}
\left\vert \frac{df_{r}}{dt}(t)\right\vert \leq & c\;t^{-\alpha -1}(1+t)^{\beta
}(1+r)^{\gamma }\\
& \times e^{-\left( {(D_{\ast }+D)t}/{2}+{%
(B_{\ast }+B)r }/{2}+{(C_{\ast }+C\lambda _{\epsilon })}%
{r ^{2}}/{8t}\right) }.
\end{split}
\end{equation*}
\end{proof}

We shall now apply estimate (\ref{Grigory'an}) in the case
of a Damek-Ricci space. The following lemma provides an initial estimate for all the derivatives of the heat kernel.

\begin{lemma}
\label{Step 0 estimates} Suppose that $S$ is a Damek-Ricci space. For all $i\in \mathbb{N}$ there is a constant $c>0$ such that
\begin{equation}\label{initial estimates}
\left\vert \frac{\partial ^{i}h_{t}}{\partial t^{i}}(r)\right\vert
\leq c\;t^{-\frac{n}{2}-i}(1+t)^{\frac{n-3}{2}},\text{ for all }t>0,\text{ }r\geq 0.
\end{equation}%

\end{lemma}

\begin{proof}
We note that according to (\ref{heat}),
\begin{equation*}
h_{t}(r)\leq c\;t^{-\frac{n}{2}}(1+t)^{\frac{n-3}{2}} \; \text{for all} \; t>0, r \geq 0.
\end{equation*}
Thus, we can apply (\ref{Grigory'an}) with
\begin{equation*}
f(t)={t^{\frac{n}{2}}}{(1+t)^{-\frac{n-3}{2}}}.
\end{equation*}

Consider the sequence of functions $f_{i},$ which is defined by
\begin{equation*}
f_{0}(t)=f(t)\text{ and }f_{i}(t)=\int_{0}^{t}f_{i-1}(s)ds, i\geq 1.
\end{equation*}
Note that $f $ is an increasing function, since $\frac{n}{2}>\frac{n-3}{2},$ thus we can invoke (%
\ref{Grigory'an}). By an induction argument we get that
\begin{equation}  \label{fi}
f_{i}(t)\geq t^{\frac{n}{2}+i}{(1+t)^{-\frac{n-3}{2}}}.
\end{equation}
Then, by (\ref{Grigory'an}) and (\ref{fi}) we get that
\begin{equation}
\left| \frac{\partial^{i} h_{t}}{\partial t^{i}}(r)\right| \leq \frac{1%
}{\sqrt{f(t)f_{2i}(t)}} \leq c t^{-\frac{n}{2}-i}(1+t)^{\frac{n-3}{2}}.
\end{equation}
\end{proof}

Note that the derivative estimates in Lemma 4 do not involve the distance $r.$ Given the estimates (\ref{heat}), and (\ref{initial estimates}), we can apply Lemma \ref{Porper} and find improved upper bounds for the $i-th$ derivative, for every $i\geq 1.$  In this way, we can refine the upper bound for the $i-th$ derivative given by Lemma 4. We can apply an inductive argument in order to obtain $\ell \in \mathbb{N}$ successive upper bounds for  the $i-th$ derivative, for every $i\geq 1.$

\begin{lemma}
\label{Technical Lemma} Suppose that $S$ is a Damek-Ricci space. Let us fix $\epsilon \in (0,1)$ and set $\lambda _{\epsilon }=\frac{%
1-\epsilon }{1+\epsilon }$. For every $i,\ell \in \mathbb{N},$ consider the sequences $\beta _{\ell}^{i},\gamma _{\ell}^{i}$ that
satisfy the iteration formulas
\begin{equation}
\begin{split}
\beta _{\ell}^{i}& =\frac{1}{2}(\beta _{\ell-1}^{i-1}+\beta _{\ell-1}^{i+1}),
\\
\gamma _{\ell}^{i}& =\frac{1}{2}(\lambda _{\epsilon}\gamma
_{\ell-1}^{i-1}+\gamma _{\ell-1}^{i+1}),
\end{split}
\label{recurrence relations}
\end{equation}%
and the initial conditions
\begin{equation}
\beta _{0}^{i}=0,\;\gamma _{0}^{i}=0,\text{ for all }\;i\geq 1, \;\beta
_{\ell}^{0}=1,\;\gamma _{\ell}^{0}=1,\text{ for all } \;\ell\geq 0.
\label{initial conditions}
\end{equation} Then, there is a constant $c>0$
\begin{equation}
\begin{split}
\left\vert \frac{\partial ^{i}h_{t}}{\partial t^{i}}(r)\right\vert \leq &
c\;t^{-\frac{n}{2}-i}(1+t)^{\frac{n-3}{2}}(1+r)^{\frac{n-1}{2}}\\
&\times e^{-\beta
_{\ell}^{i}\left( \frac{Q^2}{4} t+ \frac{Q}{2} r
\right) }e^{-\gamma _{\ell}^{i}{\frac{r^2}{4t}}},
\label{sequences estimate}
\end{split}
\end{equation}%
for all $t>0$ and $r\geq 0,$ where the constant $c$ depends on $\epsilon ,i,\ell.$ 
\end{lemma}

\begin{proof}
For every $\ell\in \mathbb{N}$ consider the following statement

$L(\ell)$: for all $i\in \mathbb{N},$ $\frac{\partial ^{i}h_{t}}{\partial t^{i}}(r) $ satisfies the estimate (\ref{sequences estimate}) and the constants $\beta _{\ell}^{i},\gamma _{\ell}^{i}$, appearing in (\ref{sequences estimate}), satisfy the iteration formulas (\ref{recurrence relations}) and the initial conditions (\ref{initial conditions}).

\smallskip
We shall then prove by
induction, that $L(\ell)$ holds for every $\ell\in \mathbb{N}$.

For $\ell=0$ we have to prove that for all $i\in \mathbb{N},$ $\frac{%
\partial ^{i}h_{t}}{\partial t^{i}}(r)$ satisfies the estimate (\ref%
{sequences estimate}) and that the constants $\beta _{0}^{i},\gamma
_{0}^{i} \geq 0$ satisfy $\beta _{0}^{i}=\gamma _{0}^{i}=0,\text{ for all }i\geq
1,\text{ and }\beta _{0}^{0}=\gamma _{0}^{0}=1.$
Indeed, from Lemma \ref{Step 0 estimates} we get that for $i\geq 1$
\begin{equation}
\left\vert \frac{\partial ^{i}h_{t}}{\partial t^{i}}(r)\right\vert
\leq c\;t^{-\frac{n}{2}-i}(1+t)^{\frac{n-3}{2}},\text{ for all }t>0,\text{ } r \geq 0.
\end{equation}%
But, $1\leq (1+r)^{\frac{n-1}{2}}.$ Thus
\begin{equation*}
\left\vert \frac{\partial ^{i}h_{t}}{\partial t^{i}}(r)\right\vert
\leq c\;t^{-\frac{n}{2}-i}(1+t)^{\frac{n-3}{2}}(1+r)^{\frac{n-1}{2}},\text{ for all }%
t>0,r \geq 0,
\end{equation*}
i.e. (\ref{sequences estimate}) holds true for all $i\geq 1$, with $\beta
_{0}^{i}=\gamma _{0}^{i}=0.$ Furthermore, from the estimates of the heat
kernel in (\ref{heat}), we obtain that
\begin{equation*}
\left\vert h_{t}(r)\right\vert \leq c\;t^{-\frac{n}{2}}(1+t)^{\frac{n-3}{2}}(1+r )^{\frac{n-1}{2}}e^{-\left( \frac{Q^2}{4}t+\frac{Q}{2}r+ \frac{r^2}{4t} \right) }.
\end{equation*}%
Thus (\ref{sequences estimate}) holds true also for $i=0$ and $\beta
_{0}^{0}=\gamma _{0}^{0}=1.$ Therefore the statement $L(0)$ holds true.

Let us assume now that $L(\ell-1)$ holds true. Thus, for all $i\in \mathbb{N}
$, there are constants $c,\beta _{\ell-1}^{i},\gamma _{\ell-1}^{i}\geq0$ such
that $\frac{\partial ^{i}h_{t}}{\partial t^{i}}(r)$ satisfies the
estimate (\ref{sequences estimate}).

We shall prove that $L(\ell)$ holds true. Indeed, from the estimates of the
heat kernel in (\ref{heat}), we have that
\begin{equation*}
|h_{t}(r)|\leq c\;t^{-\frac{n}{2}}(1+t)^{\frac{n-3}{2}}(1+r
)^{\frac{n-1}{2}}e^{-\left( \frac{Q^2}{4}t+\frac{Q}{2}r+ \frac{r^2}{4t} \right) }.
\end{equation*}%
Thus (\ref{sequences estimate}) holds true for $i=0$ with $\beta
_{\ell}^{0}=\gamma _{\ell}^{0}=1.$
For $i\geq 1$, consider the function
\begin{equation*}
f_{r}(t)=\frac{\partial ^{i-1}h_{t}}{\partial t^{i-1}}(r).
\end{equation*}%
From the validity of $L\left( \ell-1\right) $, we get that for $i-1$ and $%
i+1 $ we have that
\begin{equation*}
\begin{split}
\left\vert f_{r}(t) \right\vert= \left\vert \frac{\partial ^{i-1}h_{t}}{%
	\partial t^{i-1}}(r)\right\vert &\leq c\;t^{-\alpha }(1+t)^{\beta
}(1+r)^{\gamma }e^{-Dt-B \frac{Q}{2}r -C \frac{r^2}{4t} }, \\
\left\vert \frac{d^{2}f_{r}}{dt^{2}}(t)\right\vert =  \left\vert \frac{%
	\partial ^{i+1}h_{t}}{\partial t^{i+1}}(r)\right\vert &\leq
c\;t^{-\alpha -2}(1+t)^{\beta }(1+r)^{\gamma } e^{-D_{\ast }t-B_{\ast } \frac{Q}{2}r
	-C_{\ast } \frac{r^2}{4t}},
\end{split}
\end{equation*}

\smallskip
with
\begin{equation*}
\begin{split}
\alpha &=\frac{n}{2}+i-1, \; \beta =\frac{n-3}{2}, \; \gamma =\frac{n-1}{2},\\
 D&=B=\beta
_{\ell-1}^{i-1},\; C=\gamma _{\ell-1}^{i-1} \; \text{and}
\;D_{\ast }=B_{\ast }=\beta
_{\ell-1}^{i+1},\; C_{\ast }=\gamma _{\ell-1}^{i+1}.
\end{split}
\end{equation*}
One can verify that $\alpha>\beta$, $B \geq B_*$ and $C \geq C_*$.
Thus, by Lemma \ref{Porper} applied for the function $f_{r}(t)$, it follows
that
\begin{equation*}
\left\vert \frac{df_{r}}{dt}(t)\right\vert =\left\vert \frac{\partial
^{i}h_{t}}{\partial t^{i}}(r)\right\vert \leq
c\;t^{-\frac{n}{2}-i}(1+t)^{\frac{n-3}{2}}(1+r)^{\frac{n-1}{2}}
 e^{-\beta _{\ell
}^{i}\left( \frac{Q^2}{4t}+\frac{Q}{2}r \right)
}e^{-\gamma _{\ell }^{i}{\frac{r^2}{4t}}},
\end{equation*}%
for all $i\geq 1,$ where $\beta _{\ell }^{i}$ and $\gamma _{\ell }^{i}$
satisfy  (\ref{recurrence relations}). Thus, the statement $L(\ell )$ is
valid and the proof of the lemma is complete.
\end{proof}

\begin{remark}
The constant $c=c(i,\ell,\epsilon)$ in relation (\ref{sequences estimate})
of Lemma \ref{Technical Lemma} depends on $i, \ell$ and $\epsilon$ and it
increases to infinity (when either $i\rightarrow \infty$ or $\ell\rightarrow
\infty$ or $\epsilon\rightarrow 0$), but we only need the fact that it is
finite for fixed $i,\ell, \epsilon$.
\end{remark}

In the following Lemma, we shall prove by induction that the exponential coefficients $\beta_{\ell}^i$, $\gamma_{\ell}^i$ of Lemma 5 are convergent sequences of $\ell$. Using this fact, we shall show that these coefficients, after a sufficiently large number of iterations and $\epsilon$ sufficiently close to zero, can get arbitratily close to $1.$

\begin{lemma}
For any $i\in \mathbb{N},$
\begin{equation}  \label{sequence convergence}
\lim_{\ell\rightarrow \infty}\gamma_{\ell}^i = \left( 1 - \sqrt{1 -\lambda
{}_{\epsilon}} \right)^i \text{ and } \lim_{\ell\rightarrow
\infty}\beta_{\ell}^i = 1.
\end{equation}
\end{lemma}

\begin{proof}
We shall deal only with $\gamma_{\ell}^i$. The proof that $%
\lim_{\ell\rightarrow \infty}\beta_{\ell}^i = 1$ is similar, and we shall
omit it.

\textit{Claim 1.} For every $\ell \in \mathbb{N}$ consider the following
statement $L(\ell )$: for all $i\in \mathbb{N},$
\begin{equation}
\gamma _{\ell }^{i}\leq 1.
\end{equation}%
We shall prove by induction that $L(\ell )$ is valid for all $\ell \in
\mathbb{N}$.

For $\ell=0$ we have to prove that for all $i\in \mathbb{N},$ we have that $%
\gamma _{0}^{i}\leq 1.$ Indeed, this is a consequence of the initial
conditions $\gamma _{0}^{i}=0$ and $\gamma _{0}^{0}=1.$ Thus $L(0)$ holds
true.

Let us assume now that $L(\ell-1)$ holds true. Thus, for all $i \in \mathbb{N%
},$ we have that $\gamma_{\ell-1}^i \leq 1.$

We shall prove that $L(\ell)$ holds true. Recall that by the induction
assumption, for all $i\in \mathbb{N}$, for $i-1$ and $i+1$ we have that $%
\gamma _{\ell-1}^{i-1}\leq 1$ and $\gamma _{\ell-1}^{i+1}\leq 1$. Thus, from
(\ref{recurrence relations}) it follows that
\begin{equation*}
\gamma _{\ell}^{i}=\frac{\lambda _{\epsilon }}{2}\gamma _{\ell-1}^{i-1}+%
\frac{1}{2}\gamma _{\ell-1}^{i+1}\leq \frac{\lambda _{\epsilon }}{2}+\frac{1%
}{2}\leq 1,
\end{equation*}%
thus the statement $L(\ell)$ is valid and this completes the proof of Claim
1.

\textit{Claim 2.} For every $\ell \in \mathbb{N}$ consider the following
statement $L(\ell )$: for all $i\in \mathbb{N},$
\begin{equation}
\gamma _{\ell }^{i}\leq \gamma _{\ell +1}^{i}.
\end{equation}%
We shall prove that $L(\ell )$ is valid for all $\ell \in \mathbb{N}$. We
proceed once again by induction in $\ell \in \mathbb{N}.$

For $\ell=0$ we have to prove that for all $i\in \mathbb{N},$ $\gamma
_{0}^{i}=0\leq \gamma _{1}^{i}.$ Indeed, from (\ref{initial conditions}) it
follows that $\gamma _{0}^{i}=0\leq \gamma _{1}^{i},$ for all $i>0$ and $%
\gamma _{0}^{0}=1=\gamma _{1}^{0}.$ Therefore the statement $L(0)$ holds
true.

Let us assume now that $L(\ell-1)$ holds true, i.e. that for all $i\in
\mathbb{N},$ $\gamma _{\ell-1}^{i}\leq \gamma _{\ell}^{i}.$

We shall prove that $L(\ell)$ holds true, i.e. that for all $i\in \mathbb{N}%
, $ $\gamma _{\ell}^{i}\leq \gamma _{\ell+1}^{i}.$ Recall that by (\ref%
{recurrence relations}) we have that
\begin{equation}
\gamma _{\ell}^{i}=\frac{1}{2}(\lambda _{\epsilon }\gamma
_{\ell-1}^{i-1}+\gamma _{\ell-1}^{i+1}).  \label{itfor}
\end{equation}%
Then by the induction assumption for $i-1$ and $i+1$ we have that $\gamma
_{\ell-1}^{i-1}\leq \gamma _{\ell}^{i-1}$ and $\gamma _{\ell-1}^{i+1}\leq
\gamma _{\ell}^{i+1}.$ Hence, from (\ref{itfor}) we get that
\begin{equation*}
\gamma _{\ell}^{i}\leq \frac{1}{2}(\lambda _{\epsilon }{\gamma _{\ell}}%
^{i-1}+\gamma _{\ell}^{i+1})=\gamma _{\ell+1}^{i}.
\end{equation*}%
Thus the statement $L(\ell)$ is valid and the proof of Claim 2 is complete.

\textit{Claim 3.} For all $i\in \mathbb{N},$
\begin{equation}  \label{claim3}
\lim_{\ell\rightarrow \infty}\gamma_{\ell}^i = \left( 1 - \sqrt{1 -\lambda
{}_{\epsilon}} \right)^i .
\end{equation}

Note that by Claim 2, the sequence $\gamma _{\ell }^{i}$ is increasing in $%
\ell $ and by Claim 1, $\gamma _{\ell }^{i}$ is bounded above. Thus $%
\lim_{\ell \rightarrow \infty }\gamma _{\ell }^{i}$ exists and since $0\leq
\gamma _{\ell }^{i}\leq 1,$ then
\begin{equation*}
\lim_{\ell \rightarrow \infty }\gamma _{\ell }^{i}=\gamma _{i}\leq 1.
\end{equation*}%
Note that $\gamma _{\ell }^{0}=1$, for all $\ell \in \mathbb{N},$ thus $%
\gamma _{0}=1.$

Now, taking limits in the iteration formula (\ref{itfor}) we obtain that
\begin{equation*}
\gamma _{i}=\frac{1}{2}(\lambda _{\epsilon }\gamma _{i-1}+\gamma _{i+1}),
\end{equation*}%
thus
\begin{equation*}
\gamma _{i+1}-2\gamma _{i}+\lambda _{\epsilon }\gamma _{i-1}=0.
\end{equation*}%
This is a homogeneous linear recurrence relation with constant coefficients
and the solutions of this equation are given by
\begin{equation*}
\gamma _{i}=C_{1} \rho_{1}^{i}+C_{2} \rho_{2}^{i},\text{ \ }C_{1},C_{2}\in \mathbb{R}%
,
\end{equation*}%
where $\rho_{1}, \rho_{2}$ are the roots of the equation
\begin{equation*}
\rho^{2}-2\rho+\lambda _{\epsilon }=0.
\end{equation*}%
Thus, we conclude that
\begin{equation}
\gamma _{i}=C_{1}\left( 1-\sqrt{1-\lambda {}_{\epsilon _{1}}}\right)
^{i}+C_{2}\left( 1+\sqrt{1-\lambda {}_{\epsilon}}\right) ^{i},
\label{gammai}
\end{equation}%
for some $C_{1},C_{2}\in \mathbb{R}.$

Since $0\leq \gamma _{i}\leq 1$, we get $C_{2}=0,$ otherwise $%
\lim_{i\rightarrow \infty }\gamma _{i}=\infty $. Also, since $\gamma _{0}=1$%
, we get $C_{1}=1.$ Thus, from (\ref{gammai}) for $C_{1}=1,$ $C_{2}=0,$ we
get (\ref{claim3}) and the proof is complete.
\end{proof}

\bigskip

\textit{End of the proof of Theorem \ref{Main result}}: To complete the
proof of Theorem \ref{Main result}, note that $\lim_{\epsilon
\rightarrow 0}\left( 1-\sqrt{1-\lambda {}_{\epsilon }}\right) ^{i}=1$%
. Thus, taking $\ell\in \mathbb{N}$ sufficiently large and $\epsilon $
sufficiently close to zero, one has $\gamma _{\ell}^{i}\geq 1-\epsilon $ and
$\beta _{\ell}^{i}\geq 1-\epsilon .$ Thus, from (\ref{sequences estimate})
and (\ref{sequence convergence}) it follows that
\begin{equation*}
\left\vert \frac{\partial ^{i}h_{t}}{\partial t^{i}}(r)\right\vert \leq
c\;t^{-\frac{n}{2}-i}(1+t)^{\frac{n-3}{2}}(1+r)^{\frac{n-1}{2}}e^{-(1-\epsilon
)\left( \frac{Q^2}{t}t+ \frac{Q}{2}r + \frac{r^2}{4t}\right) }.
\end{equation*}%
Taking now into account that if $a ,b >0$, then there exists a
constant $c=c(a ,b )$ such that $x^{a}\leq ce^{b x}$ for
all $x>0,$ we conclude that for every $\epsilon >0,$ there exists a constant
$c>0$ such that
\begin{equation*}
\left\vert \frac{\partial ^{i}h_{t}}{\partial t^{i}}(x,y)\right\vert \leq
c\;t^{-\frac{n}{2}-i}e^{-(1-\epsilon )\left( \frac{Q^2}{t}t+ \frac{Q}{2}r + \frac{r^2}{4t}\right) },
\end{equation*}%
and the proof of Theorem \ref{Main result} is complete.

\bigskip

\section{Proof of Theorem \protect\ref{maximal}}

\label{Functions of the Laplacian}

In this section we apply the estimates of the derivatives of the heat
kernel.  We claim that the operators ${H}_{\sigma, max }$ and ${H}_{\sigma }$
defined in Section \ref{sectionintroduction} are bounded on $%
L^{p}(S), p \in (1, \infty),$ provided that
\begin{equation*}
\sigma < Q^2/pp' .
\end{equation*}%

We shall give only the proof for ${H}_{\sigma, max }$. The proof for ${H}%
_{\sigma }$ is similar and then omitted. Note that in this case, apart from Theorem 1, gradient estimates (\ref{grad}) are also required.

We shall consider separately the \textit{small time} maximal operator
\[
H_{\sigma, max}^{0}(f)(x)= \underset{0<t\leq 1}{\sup }e^{\sigma t}t^i | \frac{\partial^i}{\partial t^i} H_t(f)(x)|
\]
and the \textit{large time} maximal operator
\[
H_{\sigma, max}^{\infty}(f)(x)= \underset{t\geq 1}{\sup }e^{\sigma t}t^i | \frac{\partial^i}{\partial t^i}  H_t(f)(x)|.
\]
As noted in \cite[p.276]{AN1}, the whole problem comes from the component $H_{\sigma, max}^{\infty}.$

Let
\[
 {%
k}_{\sigma, max }^{\infty }(x)=\underset{t\geq 1}{\sup }e^{\sigma t}t^i | \frac{\partial^i}{\partial t^i}  h_t (x)|.
\]
Then, the component $H_{\sigma, max}^{\infty}$ can be handled by estimating
\begin{equation}\label{convinfty}
H_{\sigma, max}^{\infty}(f)(x)\leq |f| \ast  {%
k}_{\sigma, max }^{\infty }
\end{equation}
and applying the Kunze-Stein phenomenon (\ref{kunze}).

To be more precise, Theorem 1 implies the following upper bound.
\begin{lemma}
\label{global_lemma} For all $\epsilon \in (0,1) $ there exists $c>0$ such that
\begin{equation*}
\text{ }\left\vert {%
k}_{\sigma, max }^{\infty }(r)\right\vert \leq ce^{-(1-\epsilon ) \frac{Q}{2}r}e^{-(1-\epsilon )r \sqrt{\frac{Q^2}{4}-\frac{\sigma }{1-\epsilon }}},\text{ }
\end{equation*}
for all $r > 0$.
\end{lemma}

\begin{proof}
From the estimates (\ref{estimates main result}) of the derivative $\frac{%
\partial ^{i}}{\partial t^{i}}h_{t}$ provided by Theorem \ref{Main result},
we get that
\begin{align*}
\left\vert {k}_{\sigma }^{\infty }(r)\right\vert & \leq c\; \underset{t\geq
1}{\sup }\; e^{\sigma t}t^{i}\left\vert \frac{\partial ^{i}}{\partial t^{i}}%
h_{t}(r)\right\vert \\
& \leq c \; \underset{t\geq 1}{\sup }\;e^{\sigma t}e^{-(1-\epsilon )  \left( \frac{Q^2}{4}t+\frac{Q}{2}r+{
\frac{r^2}{4t}}\right) } \\
& \leq c \; \underset{t\geq 1}{\sup } \;e^{-(1-\epsilon )\frac{Q}{2}r }e^{-(1-\epsilon )\left( \left( \frac{Q^2}{4} -\frac{%
\sigma }{1-\epsilon }\right) t+ \frac{r^2}{4t}\right) }
\\
& \leq c\; e^{-(1-\epsilon ) \frac{Q}{2}r} e^{-(1-\epsilon )r
\sqrt{\frac{Q^2}{4}-\frac{\sigma }{1-\epsilon }}}.
\end{align*}
\end{proof}

Next, by (\ref{convinfty}) and (\ref{kunze}) we find
\begin{equation}
\Vert H_{\sigma, max }^{\infty }\Vert _{L^{p}(S)\rightarrow L^{p}(S)}\leq
\int_{S} {dx}\;|{k}_{\sigma, max }^{\infty }(x)|\phi _{i(\frac{1}{p}-\frac{1}{2})Q }(x)
.  \label{2KS section 5}
\end{equation}%
Then, taking into account (\ref{spherical}), we conclude that the integral condition (\ref{2KS section 5})  becomes
\begin{equation} \label{integral condition}
\quad
\begin{cases}
\lbrace \int_{0}^{1} dr\;  r^{n-1} + \int_{1}^{+\infty} dr\; e^{\frac{Q}{p}r} \rbrace k^{\infty}_{\sigma, max}(r) < +\infty, \quad \text{if} \quad 1 \leq p < 2, \\
\lbrace \int_{0}^{1} dr\; r^{n-1} + \int_{1}^{+\infty} dr\; (1+r)e^{\frac{Q}{p}r} \rbrace k^{\infty}_{\sigma, max}(r) < +\infty, \quad \text{if} \quad p=2, %
\end{cases}%
\end{equation}
\cite[p.656]{AND}. We shall show that the integrals in (\ref{integral condition}) converge. Using the estimates of ${k}_{\sigma, max }^{\infty }$ obtained in Lemma \ref%
{global_lemma} and (\ref{integral condition}), we get that
\begin{align} \label{int_sigma}
 \int_{S}  dr \; {k}_{\sigma, max }^{\infty }(x) \; \phi _{i(\frac{1}{p}-\frac{1}{2})}(x)
& \leq c \int_{0}^{1}  dr \;r^{n-1} \; e^{-(1-2\epsilon ) \frac{Q}{2}r }e^{-(1-2\epsilon ) r \sqrt{ \frac{Q^2}{4}-\frac{\sigma }{%
			1-2\epsilon }}} \notag \\
&+  \int_{1}^{+\infty}  dr \;e^{\frac{Q}{p}r} \; e^{-(1-2\epsilon ) \frac{Q}{2}r }e^{-(1-2\epsilon ) r \sqrt{ \frac{Q^2}{4}-\frac{\sigma }{%
			1-2\epsilon }}},  
\end{align}
for all $p \in (1,2)$, where we used that $(1+r)e^{-(1-\epsilon)\frac{Q}{2}r} \leq c\; e^{-(1-2\epsilon)\frac{Q}{2}r}$.

The first integral is finite, while the second integral converges, provided that
\begin{equation}
\frac{Q}{p}-(1-2\epsilon)\frac{Q}{2}-(1-2\epsilon)\sqrt{ \frac{Q^2}{4}-\frac{\sigma }{%
		1-2\epsilon }} <0.  \label{sigmaepsilon}
\end{equation}%
Choosing $\epsilon $ small enough, it follows from (\ref{sigmaepsilon}) that
the inegral in (\ref{int_sigma}) converges when
\begin{equation}
\sigma < Q^2/pp'.  \label{sigma}
\end{equation}%
Thus, by duality, $H_{\sigma, max }^{\infty }$ is bounded on $L^{p}(S),$ $p\in (1,\infty ),$
if (\ref{sigma}) holds true.

Next, it is left to show that the component $H_{\sigma, max}^{0}$ is bounded on $L^{p}(S),$ $p\in (1,\infty ).$ We split the operator $H_{\sigma, max}^{0}$ into two parts
\[
H_{\sigma, max}^{0,0}(f)(x)= \underset{0<t\leq 1}{\sup }e^{\sigma t}t^i |\frac{\partial^i}{\partial t^i}  f\ast \psi h_t(x)|
\]
and
\[
H_{\sigma, max}^{0,\infty}(f)(x)= \underset{0<t\leq 1}{\sup }e^{\sigma t}t^i |\frac{\partial^i}{\partial t^i}  f\ast (1-\psi) h_t(x)|,
\]
using a smooth cutoff fuction $\psi\in C_{c}^{\infty}(S)$ with $\psi\equiv 1$ near the origin. Then we observe that the second term $H_{\sigma, max}^{0, \infty}$ can be handled like $H_{\sigma, max}^{\infty}$ and the first term  $H_{\sigma, max}^{0,0}(f)(x)$ can be handled as in the Euclidean case (see for example \cite[p.278]{AN1}, \cite[p.670]{AND}).

\begin{remark} In a similar way one can study the boundedness of the operators (\ref{sigmamax}) and (\ref{Hsigma}) related to the Poisson operator $P_{t}=e^{-t(-\Delta_S )^{1/2}},$ as well as the boundedness of the Riesz transform $	R =\nabla (-\Delta_{S})^{-1/2}, $ already studied in \cite{AND}. However, the estimates obtained in the present paper are not sharp enough to deal with weak $L^1$ boundedness.
\end{remark}

\begin{acknowledgement}
The authors would like to thank E. Samiou for bringing to
their attention Damek-Ricci spaces, and Professor Michel Marias for his valuable
remarks.
\end{acknowledgement}

\vspace{\baselineskip}


\begin{thebibliography}{10}
\bibitem{ALX} G. Alexopoulos, \textit{Sub-Laplacians with drift on Lie groups of
		polynomial volume growth.} Memoirs of the American Mathematical
	Society \textbf{155} (2002), no. 739, 1-99.
	
\bibitem{AN1} J.-Ph. Anker, \textit{Sharp estimates for some functions of the
	Laplacian on noncompact symmetric spaces.} Duke Math. J. \textbf{65}
(1992), no. 2, 257-297.

\bibitem{AND} J.P. Anker, E. Damek, Ch. Yacoub, \textit{Spherical analysis on harmonic NA groups.} Annali Scuola Norm. Sup. di Pisa, \textbf{23} (1996), no. 4, 643-679.

\bibitem{ANJ} J.-Ph. Anker, L. Ji, \textit{Heat kernel and Green function estimates
	on noncompact symmetric spaces.} Geom. Funct. Anal. \textbf{9}
(1999), no. 6, 1035-1091.

\bibitem{ANPEVA} J.-Ph. Anker, V. Pierfelice, M. Vallarino: \textit{Schr\"odinger equations on Damek-Ricci spaces.} Commun.
Partial Differ. Equ. \textbf{36}, (2011), 976–997.

\bibitem{ANO} J.-Ph. Anker, P. Ostellari, \textit{The heat kernel on noncompact
	symmetric spaces,} Amer. Math. Soc. Transl. Ser. 2, vol. \textbf{210} (2003),
27-46.


 \bibitem{DAMEK} E. Damek, \textit{The geometry of a semi-direct extension of a Heisenberg type nilpotent group. } Colloquium Mathematicae (2) \textbf{53} (1987), 255-268.


 \bibitem{DAMRIC} E. Damek, F. Ricci,
\textit{A class of nonsymmetric harmonic Riemannian spaces.} Bull. Amer. Math. Soc. (N.S.) \textbf{27} (1992), no. 1, 139–142.


\bibitem{DAVMAN} E. B. Davies, N. Mandouvalos, \textit{Heat kernel bounds on
	hyperbolic space and Kleinian groups.} Proc. London Math. Soc. (3)
\textbf{57} (1988), no. 1, 182-208.

\bibitem{GR1} A. Grigor'yan, \textit{Heat kernel and analysis on manifolds}.
AMS/IP Studies in Advanced Mathematics, \textbf{47}. American Mathematical
Society, Providence, RI; International Press, Boston, MA, 2009.

\bibitem{GR2} A. Grigor'yan, \textit{Upper bounds of derivatives of the heat kernel
	on an arbitrary complete manifold.} J. Funct. Anal. \textbf{127}
(1995), no. 2, 363-389.
\bibitem{HE} C. Herz, Sur le ph\'{e}nom\`{e}ne de Kunze--Stein, \textit{C.
		R. Acad. Sci. Paris}, S\'{e}rie A \textbf{271} (1970), 491--493.

\bibitem{LIS} H.-Q. Li, P. Sj\"{o}gren, \textit{Weak Type (1,1) Bounds for
Some Operators Related to the Laplacian with Drift on Real Hyperbolic Spaces}%
. Potential Analysis, \textbf{46}, (2017), no.3, 463-484 .

\bibitem{LO} N. Lohou\'{e}, \textit{Estimation des fonctions de
	Littlewood-Paley-Stein sur les vari\'{e}t\'{e}s riemanniennes \`{a} courbure
	non positive.} Ann. Sci. \'{E}cole Norm. Sup. (4) \textbf{20}
(1987), no. 4, 505-544.

\bibitem{MANTSE} N. Mandouvalos, T.E. Tselepidis, \textit{Bounds of the
heat kernel derivatives on hyperbolic spaces}. Complex and harmonic
analysis, 311-326, DEStech Publ., Inc., Lancaster, PA, 2007.
\end{thebibliography}
\end{document}